\documentclass{article}
\usepackage{amsmath, amssymb, latexsym, amscd, amsthm,amsfonts,amstext}
\usepackage[mathscr]{eucal}
\usepackage{graphics, graphpap}
\usepackage{makeidx}
\usepackage{xspace}
\usepackage{array, tabularx, longtable}
\usepackage{multicol,color}
\usepackage{color}

\newtheorem{theorem}{\bf Theorem}[section]

\newtheorem{corollary}[theorem]{\bf Corollary}

\newtheorem*{theorem*}{Theorem}

\usepackage{amsmath, amssymb, latexsym, amscd, amsthm,amsfonts,amstext}
\usepackage[mathscr]{eucal}
\usepackage{graphics, graphpap}
\usepackage{makeidx}
\usepackage{xspace}
\usepackage{array, tabularx, longtable}
\usepackage{multicol,color}

\def\({\left(}
\def\){\right)}

\begin{document}

\title{Gradient estimates of Hamilton-Souplet-Zhang type for a general heat equation on Riemannian manifolds}

\author{Nguyen Thac Dung and Nguyen Ngoc Khanh}
\date{\today}
\maketitle


\begin{abstract}
The purpose of this paper is to study gradient estimates of Hamilton - Souplet - Zhang type for the following general heat equation
$$ u_t=\Delta_V u + au\log u+bu $$
on noncompact Riemannian manifolds. As its application, we show a Harnack inequality for the positive solution and a Liouville type theorem for a nonlinear elliptic equation. Our results are an extension and improvement of the work of Souplet - Zhang (\cite{SZ}), Ruan (\cite{Ruan}), Yi Li (\cite{Yili}), Huang-Ma (\cite{HM}), and Wu (\cite{Wu1}). 


\end{abstract}

\section{Introduction}
In the seminal paper \cite{liyau}, Li and Yau studied gradient estimates and Harnack inequalities for a positive solution of heat equations on a complete Riemannian manifold. Later, Li-Yau's gradient estimate was investigated by many mathematicians. Many works have been done to show generalizations and improvements of Li-Yau's results (\cite{Ruan, Wu1} and the references there in). In 1993, Hamilton introduced a different gradient estimate for a heat equation on compact Riemannian manifolds. Then, Hamilton's gradient estimate was generalized to the case of complete noncompact Riemannian manifolds. For example, see \cite{SZ} and the references there in. 

On the other hand, the weighted Laplacian on smooth metric measure spaces are of interest, recently. Recall that a smooth metric measure space is a triple $(M, g, e^{-f}dv)$, where $M$ is a Riemannian manifold with metric tensor $g, f$ is a smooth function on $M$ and $dv$ is the volume form with respect to $g$. The weighted Laplacian is defined on $M$ by 
$$ \Delta_f\cdot=\Delta\cdot - \left\langle \nabla f, \nabla \cdot\right\rangle.  $$
Here $\Delta$ stands for the Laplacian on $M$. On $(M, g, e^{-f}dv)$, the Bakry-\'{E}mery curvature $Ric_f$, and the $N$-dimensional Bakry-\'{E}mery curvature $Ric_f^{N}$ respectively are defined by  
$$ Ric_f=Ric+Hess f, \quad Ric_f^{N}=Ric_f-\frac{1}{N}\nabla f\otimes\nabla f ,$$
where  $Ric, Hess f$ are the Ricci curvature and the Hessian of $f$ on $M$, respectively. In particular, gradient Ricci solitons can be considered as smooth
metric measure spaces. Hence, the information on smooth metric measure spaces may help us to understand geometric structures of gradient Ricci solitons. Recently, X. D. Li, Huang-Ma, and Ruan investigated heat equations on smooth metric measure spaces. They have shown several results, for example gradient estimates, estimates of the heat kernel,
Harnack type inequalities, Liouville type theorems... see \cite{XDL, HM, Ruan} and the references there in.  An important generalization of the weighted Laplacian is the following operator
$$ \Delta_V\cdot=\Delta\cdot +\left\langle V, \nabla\cdot\right\rangle  $$ 
defined on Riemannian manifolds $(M, g)$. Here $\nabla$ and $\Delta$ are the Levi-Civita connenction and Laplacian with respect to metric $g$, respectively. $V$ is a smooth vector field on $M$. A natural generalization of Bakry-\'{E}mery curvature and $N$-Bakry-\'{E}mery curvature are  the following two tensors (\cite{Jost, Yili})
$$ Ric_V=Ric - \frac{1}{2}{\mathcal L}_Vg, Ric_V^{N}=Ric_V-\frac{1}{N}V\otimes V, $$
where $N>0$ is a natural number and ${\mathcal L}_V$ is the Lie derivative along the direction $V$. When $V=\nabla f$ and $f$ is a smooth function on $M$ then $Ric_V, Ric_V^{N}$ become Bakry-\'{E}mery curvature and $N$-Bakry-\'{E}mery curvature. In \cite{Yili}, Li studied gradient estimates of Li-Yau and Hamilton type for the following general heat equation
$$ u_t=\Delta_V u+au\log u $$
on compact Riemannian manifolds $(M, g)$.

In this paper, let $(M, g)$ be a Riemannian manifold and $V$ be a smooth vector field on $M$. We consider the following general heat equation
\begin{equation}\label{heat} 
u_t=\Delta_V u+au\log u+bu 
\end{equation}
where $a, b$ are functions defined on $M\times[0, +\infty)$ which are differentiable with respect to the first variable $x\in M$. Suppose $u$ is a positive solution to \eqref{heat} and $u\leq C$ for some positive constant $C$. Let $\widetilde{u}:=u/C$, then $0<\widetilde{u}\leq 1$ and $\widetilde{u}$ is a solution to 
$$ \widetilde{u}_t=\Delta\widetilde{u}+ \left\langle V,\nabla \widetilde{u}\right\rangle+a\widetilde{u}\log\widetilde{u}+\widetilde{b}\widetilde{u}, $$
where $\widetilde{b}:=\left(b+a\log C\right)$. Due to this reason, without loss of generality, we may assume $0< u\leq 1$. Our first main theorem is as follows.
 \begin{theorem}\label{main11}
 Let $M$ be a complete noncompact Riemannian manifold of dimension $n$. Let $V$ be a smooth vector field on $M$ such that $Ric_{V}\geq -K$ for some $K\geq 0$ and $|V|\leq L$ for some positive number $L$. Suppose that $a,b$ are functions of constant sign on $M\times[0, +\infty)$, moreover, $a, b$ are differentiable with respect to $x\in M$. Assume that $u$ is a  positive solution to the following general heat equation 
\begin{equation}\label{eq:2}
u_t=\Delta u+\left\langle V,\nabla u\right\rangle+au\log u+bu
\end{equation} 
on $M\times [0,+\infty)$. If $u\leq 1$, then
$$ \dfrac{|\nabla u|}{{u}}\leq \Bigg(\dfrac{1}{t^{\frac{1}{2}}}+\sup\limits_{M\times[0, +\infty)}\left\{\sqrt{2(\max\left\{0, K+a\right\}+b+|b|)}+\sqrt[4]{\dfrac{|\nabla a|^2}{2|a|}+\dfrac{|\nabla b|^2}{2|b|}}\right\}\Bigg)(1-\log u).$$
\end{theorem}
On a smooth metric measure space, in \cite{Wu}, Wu gave a gradient estimate of Souplet-Zhang type for the equation 
$$ u_t=\Delta_f u.$$
By using Brighton's proof trick \cite{Bri} and the $f$-Laplacian comparison theorem of Wei and Wylie, Wu removed the condition that $|\nabla f|$ is bounded and obtained a gradient estimate for the solution $u$. Recently, in \cite{Dung1}, the first author considered the following equation 
\begin{equation}\label{dung1}
u_t=\Delta_fu+au\log u+bu.
\end{equation}
We gave a gradient estimate of Souplet-Zhang for positive bounded solutions to \eqref{dung1} without any asumption on $|\nabla f|$ provided that $Ric_f\geq-K$. In this paper if $Ric_V^N\geq-K$, we have the following result.
 \begin{theorem}\label{main2}
 Let $M$ be a complete noncompact Riemannian manifold of dimension $n$. Let $V$ be a smooth vector field on $M$ such that $Ric_{V}^{N}\geq -K$ for some $K\geq 0$. Suppose that $a,b$ are functions of constant sign on $M\times[0,+\infty)$ and are differentiable with respect to $x$. Let $u$ be a positive solution to the following general heat equation 
$$
u_t=\Delta u+\left\langle V,\nabla u\right\rangle+au\log u+bu
$$ 
and $u\leq 1$ on $M\times [0,+\infty)$. Then
$$ \dfrac{|\nabla u|}{{u}}\leq \Bigg(\dfrac{1}{t^{\frac{1}{2}}}+\sup\limits_{M\times[0, +\infty)}\left\{\sqrt{2(\max\left\{0, K+a\right\}+b+|b|)}+\sqrt[4]{\dfrac{|\nabla a|^2}{2|a|}+\dfrac{|\nabla b|^2}{2|b|}}\right\}\Bigg)(1-\log u).$$
\end{theorem}
When $V=\nabla f$ and $a=0, b$ is a negative function, Theorem \ref{main2} recovers the main theorem in \cite{Ruan}. It is worth to notice that the inequality (1.7) in \cite{Ruan} was not completely correct. In fact, due to the proof of Theorem 1.4 in \cite{Ruan}, the function $|\nabla\ \sqrt[]{-h}|^{1/2}$ in the inequality (1.7) was evaluated at $(x_0, t_0)$. This means $|\nabla\ \sqrt[]{h}|^{1/2}$ was not computed at $(x, t)$. Therefore, the gradient estimate in the inequality (1.7) depends on $(x_0, t_0)\in B(p, 2R)\times[0, T]$. Here we used the notations given in \cite{Ruan}. However, if we assume that $\sup\limits_{M\times[0, +\infty)}|\nabla\ \sqrt[]{h}|^{1/2}<+\infty$ and replace $|\nabla\ \sqrt[]{-h}|^{1/2}$ by $\sup\limits_{M\times[0, +\infty)}|\nabla\ \sqrt[]{-h}|^{1/2}$ in the inequality (1.7), then the conclusion of Theorem 1.4 in \cite{Ruan} holds true. Hence, Theorem \ref{main11} and Theorem \ref{main2} can be considered as a generalization and improvement of the work of Souplet-Zhang, Ruan, and Y. Li.  

This paper is organized as follows. In Section 2, we prove two main theorems. Some applications are given in Section 3. In particular, we show a Harnack type inequality for the general heat equation and a Liouville type theorem for a nonlinear elliptic equation. Our results generalize a work of Huang-Ma in \cite{HM}.
\section{Gradient estimate of Hamilton - Souplet - Zhang type}
\setcounter{equation}{0}
To begin with, we restate the first main theorem. 
\begin{theorem}\label{main1}
 Let $M$ be a complete noncompact Riemannian manifold of dimension $n$. Let $V$ be a smooth vector field on $M$ such that $Ric_{V}\geq -K$ for some $K\geq 0$ and $|V|\leq L$ for some positive number $L$. Suppose that $a,b$ are functions of constant sign on $M\times[0,+\infty)$, moreover, $a, b$ are differentiable with respect to $x\in M$. Assume that $u$ is a  positive solution to the following general heat equation 
\begin{equation}\label{eq:2}
u_t=\Delta u+\left\langle V,\nabla u\right\rangle+au\log u+bu
\end{equation} 
on $M\times [0,+\infty)$. If $u\leq 1$, then
$$ \dfrac{|\nabla u|}{{u}}\leq \Bigg(\dfrac{1}{t^{\frac{1}{2}}}+\sup\limits_{M\times[0,+\infty)}\left\{\sqrt{2(\max\left\{0, K+a\right\}+b+|b|)}+\sqrt[4]{\dfrac{|\nabla a|^2}{2|a|}+\dfrac{|\nabla b|^2}{2|b|}}\right\}\Bigg)(1-\log u).$$
\end{theorem}
\begin{proof}
 Let $\square=\Delta+ \left\langle V,\nabla\right\rangle -\partial_{t}$ , $ f=\log u \leq 0$, $w=|\nabla \log(1-f)|^2$. By direct computation, we have
$$  
 \square f=\dfrac{\square u}{u}-|\nabla f|^2=-af-b-|\nabla f|^2.
 $$
Note that, in \cite{Yili}, the following $V$-Bochner-Weitzenb\"{o}ck formular is proved
$$ \dfrac{1}{2}\Delta_{V}|\nabla u|^2\geq |\nabla^2u|^2+Ric_{V}(\nabla u,\nabla u)+\left\langle\nabla\Delta_{V}u,\nabla u\right\rangle. $$
Using this inequality and the assumption $Ric_{V}\geq -K$, we have 
\begin{align}
\square w \geq &\quad 2|\nabla^2\log(1-f)|^2-2K|\nabla \log(1-f)|^2\notag\\
&\quad+2\big\langle\nabla\Delta_{V}\log(1-f),\nabla \log(1-f)\big\rangle-w_t. \label{d1}
\end{align}
On the other hand,
\begin{align}
 \Delta_{V}\log(1-f)
=&\dfrac{-\Delta_{V}f}{1-f}-w=\dfrac{-\square f-f_t}{1-f}-w=\dfrac{af+b+|\nabla f|^2-f_t}{1-f}-w\notag\\
=&\dfrac{af+b}{1-f}+\big(\log(1-f)\big)_t+(1-f)w-w \notag\\
=&\dfrac{af+b}{1-f}+\big(\log(1-f)\big)_t-fw.\label{eq:132}
\end{align}
Combining \eqref{d1} and \eqref{eq:132}, we obtain
$$
\square w \geq -2Kw+2\left\langle\nabla\left(\dfrac{af+b}{1-f}+\big(\log(1-f)\big)_t-fw\right),\nabla \log(1-f)\right\rangle-w_t.
$$
Observe that
$$ 2\big\langle\nabla \big(\log(1-f)\big)_t,\nabla \log(1-f)\big\rangle=\big(|\nabla \log(1-f)|^2\big)_t=w_t .
$$
Hence, 
\begin{align}
\square w\geq&-2Kw+2\left\langle\nabla\Big(\dfrac{af+b}{1-f}-fw\Big),\nabla \log(1-f)\right\rangle\notag\\
=&-2Kw+2\left\langle\dfrac{a\nabla f+f\nabla a+\nabla b}{1-f}-\dfrac{(af+b)(-\nabla f)}{(1-f)^2}-w\nabla f-f\nabla w,\nabla \log(1-f)\right\rangle\notag\\
=&-2Kw+2\Bigg[-aw+\left\langle\dfrac{f\nabla a+\nabla b}{1-f},\nabla \log(1-f)\right\rangle-\dfrac{af+b}{1-f}w\notag\\
 &\quad +(1-f)w^2-f\left\langle\nabla w,\nabla \log(1-f)\right\rangle  \Bigg]\notag\\
=&-2Kw-2\dfrac{af+b}{1-f}w-2aw+2\left\langle\dfrac{f\nabla a+\nabla b}{1-f},\nabla \log(1-f)\right\rangle\notag\\ &\quad +2(1-f)w^2-2f\left\langle\nabla w,\nabla \log(1-f)\right\rangle.\label{eq:135}
\end{align}
By the Schwartz
inequality, we have
\begin{align}
-\left\langle\dfrac{f\nabla a+\nabla b}{1-f},\nabla \log(1-f)\right\rangle\leq& \left|\dfrac{f\nabla a+\nabla b}{1-f}\right||\nabla \log(1-f)|\notag\\ \leq& \dfrac{-f|\nabla a|+|\nabla b|}{1-f} w^{\frac{1}{2}}\notag
\end{align}
and
$$-\left\langle\nabla w,\nabla \log(1-f)\right\rangle\leq |\nabla w||\nabla \log(1-f)|=|\nabla w|w^{\frac{1}{2}} .
$$
Combining the above inequalities and \eqref{eq:135}, it turns out that
\begin{align}
\square w\geq &-2Kw-2\dfrac{af+b}{1-f}w-2aw-2\dfrac{|\nabla b|-f|\nabla a|}{1-f}w^{\frac{1}{2}}+2(1-f)w^2+2f|\nabla f|w^{\frac{1}{2}}\notag\\
= &-2Kw -2aw +\left(-2\dfrac{b}{1-f}w-2\dfrac{|\nabla b|}{1-f}w^{\frac{1}{2}}\right)\notag\\
 &\quad +(-f)\left(-2\dfrac{-a}{1-f}w-2\dfrac{|\nabla a|}{1-f}w^{\frac{1}{2}}\right) + 2(1-f)w^2+2f|\nabla w|w^{\frac{1}{2}}.\label{eq:138}
\end{align}
Since $0<\frac{1}{1-f}\leq1$, a simple calculation shows
\begin{align}
-2\dfrac{b}{1-f}w-2\dfrac{|\nabla b|}{1-f}w^{\frac{1}{2}}=& -2\dfrac{b}{1-f}w-\dfrac{1}{1-f}2\dfrac{|\nabla b|}{\sqrt{2|b|}}\sqrt{2|b|w} \notag\\ \geq& \dfrac{1}{1-f}\Big(-2bw - \dfrac{|\nabla b|^2}{2|b|}-2|b|w\Big) \notag\\ \geq& -\dfrac{|\nabla b|^2}{2|b|}-2\big(b+|b|\big)w.\notag 
\end{align}
Similarly, since $0<\frac{-f}{1-f}\leq 1$, we have
\begin{align}
(-f)\Big(-2\dfrac{a}{1-f}w-2\dfrac{|\nabla a|}{1-f}w^{\frac{1}{2}}\Big)= &(-f)\Big(-2\dfrac{a}{1-f}w-\dfrac{1}{1-f}2\dfrac{|\nabla a|}{\sqrt{2|a|}}\sqrt{2|a|w}\Big) \notag\\ \geq& \dfrac{-f}{1-f} \Big(-2aw-\dfrac{|\nabla a|^2}{2|a|}-2|a|w\Big)\notag \\ \geq& -\dfrac{|\nabla a|^2}{2|a|}-2\big(a+|a|\big)w.\notag 
\end{align}
Hence,  the inequality \eqref{eq:138} implies
\begin{equation}\label{eq:141}
\square w \geq -2\big(K+2a+|a|+b+|b|\big)w -\dfrac{|\nabla a|^2}{2|a|}-\dfrac{|\nabla b|^2}{2|b|}+2(1-f)w^2+2f|\nabla w|w^{\frac{1}{2}}.
\end{equation}
Choose a smooth function $\eta(r)$ such that $0\leq \eta(r)\leq 1 , \eta(r)=1$ if $r\leq 1$, $\eta(r)=0$ if $r\geq 2$, and
$$ 0\geq \eta(r)^{-\frac{1}{2}}\eta(r)^{'}\geq -c_1, \quad \eta(r)^{''}\geq -c_2$$
for some $c_1,c_2\geq 0$. For a fixed point $p\in M$, let $\rho(x)=dist(p, x)$ and $\psi =\eta\left(\dfrac{\rho(x)}{R}\right)$. Therefore, 
$$
\dfrac{|\nabla\psi|^2}{\psi}=\dfrac{|\nabla\eta|^2}{\eta}=\dfrac{1}{\eta(r)}\dfrac{\big(\eta(r)^{'}\big)^2}{R^2}|\nabla \rho(x)|^2\leq \dfrac{(-c_1)^2}{R^2}=\dfrac{c_1^2}{R^2}.
$$
Since $|V|\leq L$, the Laplacian comparison theorem in \cite{Jost} implies
$$ \Delta_{V}{\rho}\leq \sqrt{(n-1)K}+ \dfrac{n-1}{\rho} +L.$$
Hence,
\begin{align}
\Delta_{V}\psi=&\dfrac{\xi(r)^{''}|\nabla \rho|^2}{R^2}+\dfrac{\xi(r)^{'}\Delta_{V}\rho}{R}\notag\\ 
\geq&\dfrac{-c_2}{R^2}+\dfrac{(-c_1)}{R}\Big[  \sqrt{(n-1)K}+ \dfrac{n-1}{\rho} +L\Big]  \notag\\
\geq&-\dfrac{R\Big[ \sqrt{(n-1)K}+ \dfrac{n-1}{R} +L\Big]c_1+c_2}{R^2}.\label{eq:804}
\end{align}
Following Calabi's argument in \cite{cal}, let $\varphi=t\psi$ and assume that $\varphi w$ obtains its maximal value on $B(p, 2R)\times[0, T]$ at some $(x, t)$, and we may assume
that $x$ is not in the locus of $p$. At $(x, t)$, we have
$$\begin{cases}
\nabla(\varphi w)=0 \\
\Delta (\varphi w)\leq 0\\
(\varphi w)_t\geq 0
\end{cases}.$$
Hence,
$$\square (\varphi w)=\Delta(\varphi w) +\big\langle V,\nabla(\varphi w)\big\rangle-(\varphi w)_t \leq 0. 
$$
Since $\square (\varphi w)=\varphi\square w+w\square\varphi + 2\left\langle\nabla w,\nabla \varphi\right\rangle $, this implies
\begin{equation}\label{eq:146}\varphi\square w+w\square\varphi + 2\left\langle\nabla w,\nabla \varphi\right\rangle\leq 0.
\end{equation}
Combining \eqref{eq:141}, \eqref{eq:146} and using the fact that $\nabla(\varphi w)=\varphi\nabla w+w\nabla \varphi =0$, we obtain 
\begin{align}
&\varphi \left\{ -2\big(K+2a+|a|+b+|b|\big)w -\dfrac{|\nabla a|^2}{2|a|}-\dfrac{|\nabla b|^2}{2|b|}+2(1-f)w^2+2f|\nabla w|w^{\frac{1}{2}}\right\} +w\square\varphi \notag\\&-2\dfrac{|\nabla\varphi|^2}{\varphi}w\leq 0.\label{eq:147}
\end{align}
Since 
 $$\begin{aligned}
2f|\nabla w|\varphi w^{\frac{1}{2}}
&=2f|\nabla \varphi| w^{\frac{3}{2}}\geq -\dfrac{f^2|\nabla \varphi|^2}{(1-f)^2\varphi}w-(1-f)^2\varphi w^2\notag\\
&\geq -\frac{|\nabla\varphi|^2}{\varphi}w -\varphi w^2.
\end{aligned}$$
Plugging this inequality into \eqref{eq:147}, we have
\begin{equation}\label{eq:252}
-2\varphi w\big(K+2a+|a|+b+|b|\big)-3\dfrac{c^2_1}{R^2}wt+\varphi w^2-\varphi\left(\dfrac{|\nabla a|^2}{2|a|}+\dfrac{|\nabla b|^2}{2|b|}\right)+w\square \varphi \leq 0.
\end{equation}
Noting that $
w\square \varphi=w\big[\Delta_{V}(t\psi)-(t\psi)_t\big]=tw\Delta_{V}\psi-\psi w$, by \eqref{eq:804} and \eqref{eq:252}, we obtain
\begin{align}
\varphi w^2+w\left\{-2\big(K+2a+|a|+b+|b|\big)\varphi +t\left(-A-\dfrac{\psi}{t}\right)\right\}-\varphi\left(\dfrac{|\nabla a|^2}{2|a|}+\dfrac{|\nabla b|^2}{2|b|}\right)\leq 0,\label{eq:83}
\end{align}
where
$$ A=\dfrac{R\Big[ \sqrt{(n-1)K}+ \dfrac{n-1}{R} +L\Big]c_1+c_2+3c_1^2}{R^2}. $$
Multiplying both sides of \eqref{eq:83} by $\varphi=t\psi$, we have at $(x, t)$
\begin{align}
\varphi^2w^2-(\varphi w)T\left\{2\big(K+2a+|a|+b+|b|\big)\psi+A+\dfrac{1}{T}\right\}-T^2\left(  \dfrac{|\nabla a|^2}{2|a|}+\dfrac{|\nabla b|^2}{2|b|}\right)\leq 0,\notag
\end{align}
where we used $0\leq\psi\leq 1, 0<t< T$. Hence,
$$
\varphi w\leq T\Big\{2\big(K+2a+|a|+b+|b|\big)\psi+A+\dfrac{1}{T}\Big\}+T\sqrt{\dfrac{|\nabla a|^2}{2|a|}+\dfrac{|\nabla b|^2}{2|b|}}.
$$
For any $(x_0, T)\in B(p, R)\times [0, T]$, we have at $(x_0, T)$
$$
w\leq \sup\limits_{M\times[0,+\infty)}\left\{2\big(\max\left\{0, K+2a+|a|\right\}+b+|b|\big)+\sqrt{\dfrac{|\nabla a|^2}{2|a|}+\dfrac{|\nabla b|^2}{2|b|}}\right\}+A+\dfrac{1}{T}.
$$
Let $R$ tend to $+\infty$, we obtain at $(x_0, T)$
\begin{align*}
 \dfrac{|\nabla u|}{{u}}\leq \Big(\dfrac{1}{T^{\frac{1}{2}}}+\sup\limits_{M\times[0, +\infty)}\left\{\sqrt{2\big(\max\{0, K+2a+|a|\}+b+|b|\big)}+\sqrt[4]{\dfrac{|\nabla a|^2}{2|a|}+\dfrac{|\nabla b|^2}{2|b|}}\right\}\Big)(1-\log u).
\end{align*}
Since $(x_0, T)$ is arbitrary, the proof is complete.
\end{proof}
Now, we give a proof of Theorem \ref{main2}.
\begin{proof}[Proof of Theorem \ref{main2}]
Since $Ric_V^N\geq-K$, the Laplacian comparison theorem in \cite{Yili} implies that  
$$ \Delta_{V}\rho\leq \sqrt{(n-1)K} coth \Big(\sqrt{\dfrac{K}{n-1}\rho}\Big)\leq \sqrt{(n-1)K}+\dfrac{n-1}{\rho}. $$
Repeating arguments in the proof of Theorem \ref{main1}, we have that in this case, the right hand side of \eqref{eq:804} does not depend on $L$. Hence, we have  
$$ A= \dfrac{\big(n-1+\sqrt{(n-1)K}R\big)c_1+c_2+3c_1^2}{R^2}.$$
The proof is complete.
\end{proof}
In particular, if $V=\nabla\phi$, $a=0$ and $b$ is a negative function on $M\times[0,+\infty)$ then we recover Ruan's main theorem in \cite{Ruan}.
\begin{corollary}(\cite{Ruan})
Let $M$ be a complete noncompact Riemannian manifold of dimension $n$ and $\phi$ be a smooth function on $M$ such that $Ric_{\phi}^{N}\geq -K$ for some $K\geq 0$. Suppose that $b$ is a non-positive function on $M\times [0,+\infty)$ and $b$ is differentiable with repect to $x$. Assume that $u$ is a positive solution of the following heat equation 
\begin{equation}\label{eq:862}
u_t=\Delta u+\left\langle \nabla\phi,\nabla u\right\rangle+bu
\end{equation} 
and $u\leq 1$ on $M\times [0,+\infty)$. Then
\begin{align*}
 \dfrac{|\nabla u|}{{u}}\leq \Big(\dfrac{1}{t^{\frac{1}{2}}}+\sqrt{2K}+\sup\limits_{M\times[0,+\infty)}{|\nabla\sqrt{-b}|}^{\frac{1}{2}}\Big)(1-\log u)
\end{align*}
provided that $\sup\limits_{M\times[0,+\infty)}{|\nabla\sqrt{-b}|}<+\infty$.
\end{corollary}
\section{Applications}
\setcounter{equation}{0}
First, we show a Harnack inequality for a general heat equation.
\begin{corollary}
 Let $M$ be a complete noncompact Riemannian manifold of dimension $n$ and $V$ be a smooth vector field on $M$ such that $Ric_{V}^{N}\geq -K$ for some $K\geq 0$. Assume that $a,b$ are functions of constant sign on $M\times [0,+\infty)$. Moreover, $a, b$ are differentiable with respect to $x\in M$. Assume that there exist $C_1, C_2>0$ satisfying 
$C_1\geq\text{max}\big\{2a+|a|,b+|b|\big\}$ and 
$$C_2\geq\text{max}\left\{\sqrt{\dfrac{|\nabla a|^2}{2|a|}},\sqrt{\dfrac{|\nabla b|^2}{2|b|}}\right\}.$$
If $u$ is a positive solution to the following general heat equation
$$
u_t=\Delta u+\left\langle V,\nabla u\right\rangle+au\log u+bu
$$
and $u\leq 1$ for all $(x,t)\in M\times (0,+\infty)$, then for any $x_1,x_2 \in M$ we have
\begin{equation}\label{eq:}u(x_2,t)\leq u(x_1,t)^{\beta}e^{1-\beta},
\end{equation}
where $\beta=\text{exp}\left(-\dfrac{\rho}{t^{\frac{1}{2}}}-(\sqrt{2(K+C_1)}+\sqrt{C_2})\rho\right)$, $\rho=\rho(x_1,x_2)$ is the distance between $x_1,x_2$.
\end{corollary}
\begin{proof}Letting $\gamma(s)$ be a geodesic of minimal length connecting $x_1$ and $x_2$, $\gamma:[0,1]\rightarrow M$, $\gamma(0)=x_2$ , $\gamma{(1)}=x_1$. Let $f=\log u$. Using Theorem \ref{main2}, we have 
\begin{align*}
\log\dfrac{1-f(x_1,t)}{1-f(x_2,t)}=&\int\limits_0^{1}\dfrac{d\log\left(1-f(\gamma(s),t)\right)}{ds} ds \\
\leq&\int\limits_0^{1}|\overset{.}{\gamma}|\dfrac{|\nabla u|}{u(1-\log u)}ds\\
\leq&\dfrac{\rho}{t^{\frac{1}{2}}}+\big(\sqrt{2(K+C_1)}+\sqrt{C_2}\big)\rho.
\end{align*}
Let $\beta=\text{exp}\Big(-\dfrac{\rho}{t^{\frac{1}{2}}}-\big(\sqrt{2(K+C_1)}+\sqrt{C_2}\big)\rho\Big)$, then the above inequality implies
$$\dfrac{1-f(x_1,t)}{1-f(x_2,t)}\leq \dfrac{1}{\beta}.
$$
Hence,
$$u(x_2,t)\leq u(x_1,t)^{\beta}e^{1-\beta}.$$
The proof is complete.
\end{proof}
\begin{corollary}\label{corolla1}
 Let $M$ be a complete noncompact Riemannian manifold of dimension $n$ and $V$ be a smooth vector field on $M$ such that $Riv_{V}^{N}\geq -K$ for some $K\geq 0$. Suppose that $a,b$ are negative real numbers and the positive solution $u$ to the heat equation
$$
u_t=\Delta u+\left\langle V,\nabla u\right\rangle+au\log u+bu
$$
satisfying $u\leq 1$. Then 
\begin{equation}\label{eq:1}\dfrac{|\nabla u|}{u}\leq \Big(\dfrac{1}{t^{\frac{1}{2}}}+\sqrt{2\max\{0, K+a\}}\Big)(1-\log u).
\end{equation}
\end{corollary}
\begin{proof}
By \eqref{eq:141} we have
\begin{equation}\label{eq:883}
\square w \geq -2\big(K+2a+|a|+b+|b|\big)w -\dfrac{|\nabla a|^2}{2|a|}-\dfrac{|\nabla b|^2}{2|b|}+2(1-f)w^2+2f|\nabla w|w^{\frac{1}{2}}.
\end{equation}
If $a\leq0$ then \eqref{eq:883} implies
$$\square w \geq -2\big(\max\{0, K+a\}+b+|b|\big)w -\dfrac{|\nabla b|^2}{2|b|}+2(1-f)w^2+2f|\nabla w|w^{\frac{1}{2}}.$$ 
Therefore, the conclusion of Theorem \ref{main2} can be read as
$$ \dfrac{|\nabla u|}{{u}}\leq \Bigg(\dfrac{1}{t^{\frac{1}{2}}}+\sqrt{2(\max\{0, K+a\}+b+|b|)}+\sqrt[4]{\dfrac{|\nabla b|^2}{2|b|}}\Bigg)(1-\log u).$$
Since $b$ is a negative real number, we are done.
\end{proof}
Now we can show a Liouville type result.
\begin{corollary}\label{liouville}
 Let $M$ be a complete noncompact Riemannian manifold and $V$ be a smooth vector field on $M$ such that $Ric_{V}^{N}\geq -K$ for some $K\geq 0$. Suppose that $a,b$ are negative real numbers, $a\leq -K$. If $u$ is a positive solution to the following general elliptic equation 
$$
\Delta u+\left\langle V,\nabla u\right\rangle+au\log u+bu=0
$$
and $u\leq 1$, then $u\equiv e^{-\frac{b}{a}}$. 
\end{corollary}
\begin{proof}
Since $a\leq-K$, we have $\max\{0, K+a\}=0$. Hence, let $t$ tend to $+\infty$ in \eqref{eq:1}, we obtain
$$\dfrac{|\nabla u|}{u}\leq 0.
$$
This implies $u$ must be a constant. Therefore $u=e^{-\frac{b}{a}}$.
\end{proof}
We note that in \cite{HM}, Huang and Ma proved the following Liouville type theorem.
\begin{corollary}(\cite{HM})\label{Hm}
Let $(M, g)$ be an $n$-dimensional complete noncompact Riemannian manifold with $Ric\geq-K$, where $K\geq0$ is a constant. Suppose that $u$ is a bounded solution defined on $M$ to 
$$ \Delta u+au\log u=0 $$
with $a<0$. If $a\leq-K$, then $u\equiv 1$ is a constant. 
\end{corollary}
Now, suppose that $u$ is a positive solution to 
$$ \Delta u+\left\langle V,\nabla u\right\rangle+au\log u+bu=0 $$
and $u\leq C$, where $a, b\leq 0$ are constants, $a\leq-K$. We may assume $C\geq1$, then $\widetilde{u}:=u/C\leq 1$ is a positive solution to 
$$ \widetilde{u}_t=\Delta\widetilde{u}+ \left\langle V,\nabla \widetilde{u}\right\rangle+a\widetilde{u}\log\widetilde{u}+\widetilde{b}\widetilde{u}, $$
where $\widetilde{b}:=\left(b+a\log C\right)$. If $Ric_V^N\geq-K$, by Corollary \ref{liouville}, we have that 
$\widetilde{u}=exp\left(-\frac{b}{a}-\log C\right)$, so $u=C exp\left(-\frac{b}{a}-\log C\right)$. If $b=0$ then $u\equiv 1$. Therefore, we give another proof of Huang and Ma's result. Moreover, it is easy to see that Corollary \ref{liouville} is a generalization of Corollary \ref{Hm}. 

\vskip 0.3cm
\noindent
{\bf Acknowledment. }We would like to express our deep thanks to an anonymous referee for his careful reading and helpful suggestions to improve this paper. The first author was supported in part by NAFOSTED under grant number 101.02-2014.49. A part of this paper was written during a his stay at the Vietnam Institute for Advance Study in Mathematics (VIASM). He would like to express his sincere thanks to the staff there for the excellent working condition and the financial support.

\vskip0.3cm
\noindent
Department of Mathematics, Mechanics and Informatics (MIM),\\
Hanoi University of Sciences (HUS-VNU),\\
No. 334, Nguyen Trai Road,\\
Thanh Xuan, Hanoi, Vietnam.\\
\textit{Email: dungmath@gmail.com (Nguyen Thac Dung)}\\
\textit{Email: khanh.mimhus@gmail.com (Nguyen Ngoc Khanh)}
\end{document}